\documentclass{amsart}
\usepackage[margin=1.0in]{geometry}
\usepackage{amsthm, amsmath, amssymb, mathtools, mathrsfs}
\usepackage{trimclip, enumitem, adjustbox}
\usepackage{tikz}
\usepackage{tikz-qtree}
\usepackage{aliascnt}
\usepackage{hyperref}
\usepackage{breakurl}
\usepackage{array, longtable}
\usepackage{paracol}
\usepackage{bussproofs}
\usepackage{csquotes}
\usepackage{version}

\usepackage[style = numeric, sorting = nyt]{biblatex}
\makeatletter
\def\LT@start{%
\let\LT@start\endgraf
\endgraf\penalty\z@\vskip\LTpre
\dimen@\pagetotal
\advance\dimen@ \ht\ifvoid\LT@firsthead\LT@head\else\LT@firsthead\fi
\advance\dimen@ \dp\ifvoid\LT@firsthead\LT@head\else\LT@firsthead\fi
\advance\dimen@ \ht\LT@foot
\dimen@ii\vfuzz
\vfuzz\maxdimen
\setbox\tw@\copy\z@
\setbox\tw@\vsplit\tw@ to \ht\@arstrutbox
\setbox\tw@\vbox{\unvbox\tw@}%
\vfuzz\dimen@ii
\advance\dimen@ \ht
\ifdim\ht\@arstrutbox>\ht\tw@\@arstrutbox\else\tw@\fi
\advance\dimen@\dp
\ifdim\dp\@arstrutbox>\dp\tw@\@arstrutbox\else\tw@\fi
\advance\dimen@ -\pagegoal
\ifdim \dimen@>\z@\unskip\vfil\break\fi
\global\@colroom\@colht
\ifvoid\LT@foot\else
\advance\vsize-\ht\LT@foot
\global\advance\@colroom-\ht\LT@foot
\dimen@\pagegoal\advance\dimen@-\ht\LT@foot\pagegoal\dimen@
\maxdepth\z@
\fi
\ifvoid\LT@firsthead\copy\LT@head\else\box\LT@firsthead\fi
\output{\LT@output}}
\makeatother

\emergencystretch=\maxdimen
\usepackage[textsize=footnotesize]{todonotes}

\usepackage{xcolor}
\definecolor{dblue}{rgb}{0,0,0.70}
\hypersetup{
	unicode=true,
	colorlinks=true,
	citecolor=dblue,
	linkcolor=dblue,
	anchorcolor=dblue
}

\makeatletter
\expandafter\g@addto@macro\csname th@plain\endcsname{%
	\thm@notefont{\bfseries}
}%
\expandafter\g@addto@macro\csname th@remark\endcsname{%
	\thm@headfont{\bfseries}
}%
\makeatother

\newtheorem{theorem}{Theorem}[section]	
\newtheorem*{theorem*}{Theorem}

\newaliascnt{lemma}{theorem}
\newtheorem{lemma}[lemma]{Lemma}
\aliascntresetthe{lemma}
\newtheorem*{lemma*}{Lemma}

\newaliascnt{proposition}{theorem}
\newtheorem{proposition}[proposition]{Proposition}
\newtheorem*{proposition*}{Proposition}
\aliascntresetthe{proposition}

\newaliascnt{corollary}{theorem}
\newtheorem{corollary}[corollary]{Corollary}
\aliascntresetthe{corollary}

\theoremstyle{remark}

\newaliascnt{remark}{theorem}

\aliascntresetthe{remark}
\newaliascnt{question}{theorem}

\aliascntresetthe{question}

\newaliascnt{conjecture}{theorem}

\aliascntresetthe{conjecture}

\newtheorem*{question*}{Question}

\newaliascnt{definition}{theorem}
\newtheorem{definition}[definition]{Definition}
\aliascntresetthe{definition}

\newaliascnt{example}{theorem}

\aliascntresetthe{example}

\newaliascnt{convention}{theorem}

\aliascntresetthe{convention}

\newaliascnt{claim}{theorem}

\aliascntresetthe{claim}

\makeatletter
\def\l@subsection{\@tocline{2}{0pt}{1pc}{5pc}{}} \def\l@subsection{\@tocline{2}{0pt}{2pc}{6pc}{}}
\makeatother
\usepackage[style = numeric, sorting = nyt]{biblatex}
\title{The proof-theoretic strength of Constructive Second-order set theories}
\author{Hanul Jeon}
\email{ \href{mailto:hj344@cornell.edu}{hj344@cornell.edu}}
\urladdr{ \href{https://hanuljeon95.github.io}{https://hanuljeon95.github.io} }
\address{Department of Mathematics, Cornell University, Ithaca, NY 14853}

\makeatletter
\DeclareRobustCommand\widecheck[1]{{\mathpalette\@widecheck{#1}}}
\def\@widecheck#1#2{%
    \setbox\z@\hbox{\m@th$#1#2$}%
    \setbox\tw@\hbox{\m@th$#1%
       \widehat{%
          \vrule\@width\z@\@height\ht\z@
          \vrule\@height\z@\@width\wd\z@}$}%
    \dp\tw@-\ht\z@
    \@tempdima\ht\z@ \advance\@tempdima2\ht\tw@ \divide\@tempdima\thr@@
    \setbox\tw@\hbox{%
       \raise\@tempdima\hbox{\scalebox{1}[-1]{\lower\@tempdima\box
\tw@}}}%
    {\ooalign{\box\tw@ \cr \box\z@}}}
\makeatother

\mathchardef\mhyphen="2D
\newcommand{\restricts}{\mkern-5mu\upharpoonright}

\newcommand{\rrarrows}{\rightrightarrows}

\newcommand{\lrlrarrows}{\mathrel{\mathrlap{\leftleftarrows}{\rightrightarrows}}}
\newcommand{\lr}{\leftrightarrow}

\newcommand{\ran}{\operatorname{ran}}

\newcommand{\CST}{\mathsf{CST}}

\newcommand{\IZF}{\mathsf{IZF}}
\newcommand{\CZF}{\mathsf{CZF}}
\newcommand{\ZF}{\mathsf{ZF}}

\newcommand{\IGB}{\mathsf{IGB}}
\newcommand{\CGB}{\mathsf{CGB}}
\newcommand{\GB}{\mathsf{GB}}

\newcommand{\IKM}{\mathsf{IKM}}
\newcommand{\CKM}{\mathsf{CKM}}
\newcommand{\KM}{\mathsf{KM}}

\newcommand{\KP}{\mathsf{KP}}

\newcommand{\bbN}{\mathbb{N}}

\newcommand{\bfp}{\mathbf{p}}
\newcommand{\bfq}{\mathbf{q}}
\newcommand{\sfN}{\mathsf{N}}
\newcommand{\sfV}{\mathsf{V}}
\newcommand{\frakt}{\mathfrak{t}}

\newcommand{\lag}{\langle}
\newcommand{\rag}{\rangle}
\newcommand{\len}{\operatorname{len}}

\newcommand{\WF}{\operatorname{WF}}

\newcommand{\immd}{\operatorname{immd}}
\newcommand{\ext}{\operatorname{ext}}

\newcommand{\Set}{\operatorname{Set}}

\begin{document}

\begin{abstract}
    In this paper, we define constructive analogues of second-order set theories, which we will call $\mathsf{IGB}$, $\mathsf{CGB}$, $\mathsf{IKM}$, and $\mathsf{CKM}$. Each of them can be viewed as 
    $\mathsf{IZF}$- and $\mathsf{CZF}$-analogues of G\"odel-Bernays set theory $\mathsf{GB}$ and Kelley-Morse set theory $\mathsf{KM}$.
    We also provide their proof-theoretic strengths in terms of classical theories, and we especially prove that $\mathsf{CKM}$ and full Second-Order Arithmetic have the same proof-theoretic strength. 
\end{abstract}

\maketitle

\section{Introduction}
There are two ways to formulate constructive versions of Zermelo-Frankel set theory $\ZF$, namely \emph{Intuitionistic Zermelo-Frankel set theory} $\IZF$ and \emph{Constructive Zermelo-Frankel set theory} $\CZF$.
Their strengths and metamathematical properties are broadly studied by Aczel, H. Friedman, Rathjen, and others. 
While $\ZF$ is a first-order set theory--the only objects in $\ZF$ are sets--there is a natural demand for `large collections' that are not formally sets.
These objects are called \emph{classes}, and \emph{second-order set theories} like \emph{G\"odel-Bernays set theory} $\GB$ and \emph{Kelley-Morse set theory} $\KM$ supports classes as objects.

It is natural to ask whether we have a constructive version of second-order set theories.
The main goal of this paper is to formulate four possible second-order set theories named $\IGB$, $\IKM$, $\CGB$, and $\CKM$, and to provide their proof-theoretic strength. $\IGB$ and $\CGB$ are $\GB$-like second-order set theories relating to $\IZF$ and $\CZF$, respectively. (See \autoref{Section: Set Theory} for a precise definition.) Since $\GB$ is conservative over $\ZF$, we may expect the same between $\IGB$ and $\IZF$, and $\CGB$ and $\CZF$ respectively. It turns out that a proof-theoretic argument for the conservativity of $\GB$ over $\ZF$ carries over (cf. \cite[Problem \S16.11]{Takeuti1987}), so we obtain the following result:
\begin{proposition}\pushQED{\qed}
    $\IGB$ is conservative over $\IZF$ and $\CGB$ is conservative over $\CZF$. As a corollary, $\IGB$ and $\IZF$ have the same proof-theoretic strength, and the same holds for $\CGB$ and $\CZF$. \qedhere 
\end{proposition}

Notice that $\KM$ is also a second-order extension of $\ZF$, as it is obtained from $\GB$ by adding class comprehension rules for arbitrary second-order formulas. $\KM$ is a non-conservative extension of $\ZF$ since $\KM$ proves the consistency of $\ZF$. We can expect the same for the constructive analogues $\KM$, in other words, $\IKM$ has a higher proof-theoretic strength than $\IZF$, and the same between $\CKM$ and $\CZF$.
We can see that the double-negation translation between $\IZF$ and $\ZF$ can carry over a second-order setting. We may directly work with H. Friedman's two-stage double-negation translation via non-extensional set theories, or forcing with the double-negation topology adopted in \cite{JeonMatthews2022}.
The same machinery works between $\IKM$ and $\KM$, which shows the proof-theoretic strength of $\IKM$:
\begin{corollary}\pushQED{\qed}
    The proof-theoretic strength of $\IKM$ is that of $\KM$. \qedhere 
\end{corollary}

What can we say about the proof-strength of $\CKM$? It obviously lies between $\CZF$ and $\KM$, but even the gap between $\CZF$ and $\ZF$ is very large: It is known by \cite{Rathjen1993} that the proof-theoretic strength of $\CZF$ is equal to that of $\KP$, and $\ZF$ proves not only the consistency of $\KP$ but its strengthenings like $\KP$ with large countable ordinals like admissible ordinals. The main result in this paper is the optimal proof-theoretic strength of $\CKM$ in terms of well-known classical theories:
\begin{theorem*}
    The proof-theoretic strength of $\CKM$ is equal to that of full Second-Order Arithmetic.
\end{theorem*}
Its proof follows from a modification of Lubarsky's proof \cite{Lubarsky2006SOA} of the equiconsistency between full Second-Order Arithmetic and $\CZF$ with the full Separation. 

Before starting the main content, we will briefly overview the structure of the paper. In \autoref{Section: Set Theory}, we review a historical account and basic notions of constructive set theory, then define the constructive versions of second-order set theories. In \autoref{Section: Prelim}, we introduce preliminary notions necessary to establish the proof-theoretic strength of $\CKM$. In \autoref{Section: Strength of CKM}, we prove that the proof-theoretic strength of $\CKM$ is that of full Second-Order Arithmetic.

\section*{Acknowledgement}
The initial idea of this paper stemmed from the final report for the proof theory class taught by James Walsh, and the author would like to thank him for allowing the author to write the final report about constructive set theory. The author also would like to thank Robert Lubarsky for helpful comments, and Justin Moore for improving my writing.

\section{Constructive set theory}
\label{Section: Set Theory}
In this section, let us briefly review constructive set theories, and define constructive second-order set theories.
The standard references for constructive set theory would be \cite{AczelRathjen2001} and \cite{AczelRathjen2010}, but these might be hard to read for unmotivated readers. Thus we try to include a historical account and explanations for notions appearing in $\CZF$, which might help the reader to work on constructive set theory.

\subsection{Historical account}
Various constructive systems appeared after Bishop's paramount `Constructive analysis' \cite{BishopBridges1985}, which showed constructive mathematics is not remote from classical mathematics. Unfortunately, Bishop formulated his concepts informally and he did not present any foundational system for his analysis, although he remained some remnants of some criteria that his system must satisfy. For example, he mentioned in his book \cite{BishopBridges1985} that the Axiom of Choice follows from `by the very meaning of the existence.'%
\footnote{The author learned later that this sentence can be misleading; Bishop claimed we have a choice \emph{operation} that may not respect the equality, and not a choice function. See \cite{McCartyShapiroKlev2023Choice} for a more discussion related to the axiom of choice in constructive mathematics.}
Many consequent attempts at formulating a foundation for Bishop's system appeared, including H. Friedman's Intuitionistic Zermelo-Frankel set theory, Feferman's Explicit Mathematics, Myhill's $\CST$, or Martin-L\"of's type theory.

Since the set theory successfully settled the foundation crisis of mathematics, it would be natural to consider the constructive analogue of set theory as a foundation. The standard set theory for classical mathematics is \emph{Zermelo-Frankel set theory} $\ZF$, so the constructive analogue of set theory should resemble $\ZF$. But classical $\ZF$ is deemed to be \emph{impredicative} due to unrestricted quantification in the axiom of Separation and Powerset, which allows us to construct objects of a certain rank by referring to objects of higher ranks.\footnote{\cite{Myhill1973IZF} and \cite{Crosilla2000ST} claimed Powerset is impredicative because it does not give any rule to generate elements of a power set of a given set, and it instead assumes the universe of all sets and collects all possible subsets of the given set. See \cite{Crosilla2000ST} for a more discussion about the impredicativity of Separation and Powerset.}
However, going back to Bishop's constructive analysis that was an initial motivation for constructive foundations, Petrakis (\cite{Petrakis2020Habilitation}, Chapter 1) stated that ``... Bishop elaborated a version of dependent type theory with one universe, in order to formalise $\mathsf{BISH}$.''
That is, Bishop tried to formulate his constructive analysis over a type theory, which is deemed as predicative, so a set-theoretic foundation for constructive analysis should be predicative.
% Due to the predicative nature of dependent type theory, we would imply from Bishop's attempt that an appropriate set-theoretic foundation for constructive analysis should be predicative. But how can we test the predicative nature of set theories? One way to do it is to provide a reduction to a previously known predicative theory. 
Then what is the constructive predicative set theory resembling $\ZF$? Myhill \cite{Myhill1973IZF} and H. Friedman \cite{Friedman1973Intuitionistic} studied a constructive version of set theory now called \emph{Intuitionistic Zermelo-Frankel set theory} $\IZF$. It turns out by H. Friedman \cite{Friedman1973} that the proof-theoretic strength of $\IZF$ is equal to that of $\ZF$, meaning that $\IZF$ is highly impredicative. This rules out $\IZF$ as a predicative foundation for constructive analysis.
Myhill \cite{Myhill1975} defined a version of `predicative' constructive set theory named $\CST$ as an attempt to provide a foundation for Bishop's constructive analysis \cite{BishopBridges1985}. However, Myhill's formulation is not quite close to that of $\ZF$: for example, $\CST$ contains atoms for natural numbers. 
Years later, Aczel formulated \emph{Constructive Zermelo-Frankel set theory} $\CZF$ and its type-theoretic interpretation in Martin-L\"of type theory by modifying the type-theoretic interpretation of ordinals that appeared in the thesis of his doctoral student Leversha \cite{LevershaPhD}.
$\CZF$ is more $\ZF$-like than $\CST$, and it can be interpreted in (Martin-L\"of's) type theory (See \cite{Aczel1978}, \cite{Aczel1982}, \cite{Aczel1986}, or \cite{Wehr2021}), so $\CZF$ has an reduction to a predicative theory. Therefore, $\CZF$ satisfies the mentioned criteria for the set-theoretic foundation for constructive analysis. 
$\CZF$ with Choice is an impredicative theory whose consistency strength is far higher than that of the full second-order arithmetic,\footnote{By Diaconescu's theorem, Choice over $\CZF$ implies Excluded middle for $\Delta_0$-formulas, which implies Powerset. Rathjen proved in \cite{Rathjen2012Power} that $\CZF$ with powerset is equiconsistent with Power Kripke-Platek set theory, which is stronger than the full second-order arithmetic.} 
but $\CZF$ with weaker forms of Choice like \emph{Countable Choice} or \emph{Dependent Choice} can still be regarded predicative, which are known to be enough to develop the foundations of real analysis. (See \cite{Aczel1982} or \cite{Wehr2021} for the type-theoretic interpretation of $\CZF$ with Countable Choice or Dependent Choice.) For a detailed discussion of historical aspects of constructive set theories, the reader might refer to \cite{Crosilla2000ST}.

\subsection{First-order set theory}
Let us start with the definition of $\IZF$ that is closer to that of $\ZF$ than $\CZF$:
\begin{definition}
    $\IZF$ is an intuitionistic first-order theory over the language $\langle\in\rangle$ with the following axioms: Extensionality, Pairing, Union, Infinity, $\in$-Induction, Separation, Collection, and Powerset.
\end{definition}
Unlike $\ZF$, $\IZF$ uses Collection instead of Replacement as an axiom because there seems no good way to derive the full strength of $\ZF$ from $\IZF_R$, $\IZF$ with Replacement instead of Collection.
In fact, H. Friedman and \v{S}\v{c}edrov proved in \cite{FriedmanScedrov1985} that $\IZF$ has more provably recursive functions than $\IZF_R$. Thus $\IZF_R$ is `weaker' than $\IZF$ in some sense, although it is open whether $\IZF$ proves the consistency of $\IZF_R$.

$\CZF$ appeared first in Aczel's paper \cite{Aczel1978} about the type-theoretic interpretation of set theory.
\begin{definition}
    $\CZF$ is an intuitionistic first-order theory over the language $\langle\in\rangle$ with the following axioms: Extensionality, Pairing, Union, Infinity, $\in$-Induction with
    \begin{itemize}
        \item $\Delta_0$-Separation: $\{x\in a\mid \phi(x,p)\}$ exists if $\phi$ is a bounded formula.
        \item Strong Collection: If $\forall x\in a\exists y\phi(x,y)$, then there is $b$ such that
        \begin{equation*}
            \forall x\in a\exists y\in b \phi(x,y)\text{ and } \forall y\in b\exists x\in a\phi(x,y).
        \end{equation*}
        \item Subset Collection: For any $a$ and $b$, we can find $c$ such that for any $u$, if $\forall x\in a\exists y \in b \phi(x,y,u)$, then we can find $d\in c$ such that
        \begin{equation*}
            \forall x\in a\exists y\in d \phi(x,y,u)\text{ and } \forall y\in d\exists x\in a\phi(x,y,u).
        \end{equation*}
    \end{itemize}
\end{definition}
The formulation of $\CZF$ may appear to be a bit strange since the choice of axioms is less motivated (especially, Subset Collection). The following notions could be helpful in approaching these two axioms:
\begin{definition}
    Let $A$ and $B$ be classes.
    \begin{itemize}
        \item A relation $R\subseteq A\times B$ is a \emph{multi-valued function} if $\forall x\in A\exists y \langle x,y\rangle\in R$. In this case, we use the notation $R\colon A\rrarrows B$ and we call $A$ a \emph{domain of $R$}.
        \item If $R\colon A\rrarrows B$ further satisfies $\forall y\in B\exists x\in A \langle x,y\rangle\in R$, then we call $B$ a \emph{subimage of $R$} and use the notation $R\colon A\lrlrarrows B$.
    \end{itemize}
\end{definition}
In the $\CZF$-world, multi-valued functions are more fundamental than functions and subsets. The notion of subsets is associated with power sets, which are viewed as impredicative. But we may ask why functions are not `fundamental' like multi-valued functions. The reason comes from the type-theoretic interpretation of $\CZF$.
When we interpret Strong Collection over the type interpretation, for example, we extract a choice function from the statement $\forall x\in a\exists y \phi(x,y)$. However, the choice function $f$ we get is \emph{not necessarily} extensional (or put differently, $f$ is an operation and not a function): The choice `function' $f$ depends on the `type-theoretic expression' of sets $x\in a$. We can still form a set-sized `mapping' $G=\{\lag x,f(x)\rag \mid x\in a\}\subseteq \{\lag x,y\rag\mid x\in a\land \phi(x,y)\}$, but $G$ as a graph is no longer functional since a given set $x\in a$ may have different type-theoretic expressions, and $f(x)$ depends on that expression. Thus $G$ ends up being a multi-valued function. A similar situation also appears around interpreting Subset Collection.

We can see that if $R$ is functional, that is, for each $x$ there is at most one $y$ such that $\langle x,y\rangle\in R$, then a subimage must be an image of $R$. However, unlike images, subimages of a given multi-valued function need not be unique.
%\begin{example}
%    Consider $A=\{0,1\}$ and let
%    \begin{equation*}
%        R=\{\la 0,0\ra, \la 1,0\ra, \la 1,1\ra\}.
%    \end{equation*}
%    Then both of $\{0\}$ and $\{0,1\}$ are subimages of $R\colon A\rrarrows A$.
%\end{example}
Under this terminology, we can formulate Strong Collection and Subset Collection as variations of Replacement and Powerset in $\ZF$. For Strong Collection, let us recall that Replacement states

\begin{displayquote}
For every set $a$ and a class function $F\colon a\to V$, $F$ has an image $b$ of $F$.
\end{displayquote}

By replacing `function' and `image' with `multi-valued function' and `subimage,' we get the statement of Strong Collection:

\begin{displayquote}
For every set $a$ and a class \emph{multi-valued function} $F\colon a\rrarrows V$, $F$ has a \emph{subimage} $b$ of $F$.
\end{displayquote}

The case for Subset Collection requires more work because it requires functional formulation for Powerset.
The reason we take a functional formulation instead of a usual formulation of Powerset is the axiom of Powerset itself is impredicative as Myhill \cite{Myhill1975} pointed out. On the other hand, the notion of functions is not impredicative since ``a mapping or function is a rule, a finite object which can actually be given'' according to Myhill \cite[p351]{Myhill1975}.

\begin{lemma}[$\ZF$ without Powerset] \label{Lemma: Sect2-Powerset-00}
    Powerset is equivalent to the following claim: for given sets $a$ and $b$ we can find a set $c$ such that if $f\colon a\to b$ is any (set or class) function, then $c$ contains the image of $f$.
\end{lemma}
\begin{proof}
    The forward direction is clear. For the other direction, let $c$ be the set containing all images of functions from $a$ to $a\times 2$. For each $u\subseteq a$, consider the function whose existence is guaranteed by Replacement:
    \begin{equation*}
        f_u(x)=\begin{cases}
        \langle x,1\rangle & \text{if }x\in u,\\
        \langle x,0\rangle & \text{if }x\notin u.
        \end{cases}
    \end{equation*}
    Then $\ran f_u = (a\setminus u)\times \{0\} \cup u\times \{1\}$.
    Hence the set
    \begin{equation*}
        \{ \{x\in a : \langle x,1\rangle\in u\} \mid u\in c\},
    \end{equation*}
    which exists by Replacement, contains all subsets of $a$.
\end{proof}

Since the set of all images of functions $f\colon a\to b$ is an example of a collection of definable class functions, we can further strengthen the previous lemma as follows:
\begin{lemma}[$\ZF$ without Powerset]
    Powerset is equivalent to the following statement: for given sets $a,b$ and a class family of classes $\{F_u \mid u\in V\}$ defined by $F_u = \{\lag x,y\rag \mid\phi(x,y,u)\}$,
    we can find a set $c$ such that if $F_u\colon a\to b$ is a function, then $\ran F_u\in c$.
\end{lemma}
\begin{proof}
    The forward direction is clear.
    The converse direction follows from \autoref{Lemma: Sect2-Powerset-00} since we can take $\{F_u\mid u\in V\}=\{f\mid f\colon a\to b\}$.
\end{proof}

By modifying the above lemma, we can notably get the statement of Subset Collection:
\begin{displayquote}
    For given sets $a,b$ and a class family of classes $\{F_u \mid u\in V\}$ defined by $F_u = \{\lag x,y\rag \mid\phi(x,y,u)\}$,
    we can find a set $c$ such that if $F_u\colon a\rrarrows b$ is a \emph{multi-valued function}, then $c$ contains a \emph{subimage} of $F_u$.
\end{displayquote}

One expected consequence of Subset Collection is that it implies the \emph{function set} exists:
\begin{proposition}
    $\CZF$ proves for each sets $a$ and $b$, ${}^ab=\{f\mid f\colon a\to b\}$ exists.
\end{proposition}
\begin{proof}
    Consider the following class families $\{F_u\mid u\in V\}$ defined as follows: 
    \begin{equation*}
        F_u=\{\langle x,\langle x,y\rangle\rangle \mid \langle x,y\rangle\in u\cap (a\times b) \}.
    \end{equation*}
    We can see that if $u$ is a function from $a$ to $b$, then $F_u$ is a function from $a$ to $a\times b$ and the image of $F_u$ is $u$.
    Now let $c$ be a set of subimages of $\{F_u\mid u\in V,\, F_u\colon a\rrarrows b\}$. If $u\colon a\to b$ is a function, then it is the image of $F_u$, which is the unique subimage of $F_u$ since $F_u$ is a function. Hence $u\in c$. Therefore, by Bounded Separation, we have
    \begin{equation*}
        {}^ab=\{u\in c\mid u\colon a\to b\}. \qedhere 
    \end{equation*}
\end{proof}

However, it is known that $\CZF$ does not prove the existence of a powerset: Working over $\CZF$ with Powerset, we can construct $V_{\omega+\omega}$, which is a model of Intuitionistic Zermelo set theory $\mathsf{IZ}$. Furthermore, \cite{GambinoThesis} showed that $\mathsf{IZ}$ interprets classical Zermelo set theory by Friedman's double-negation interpretation. On the other hand, \cite{Rathjen1993} proved that the proof-theoretic strength of $\CZF$ is equal to that of $\KP$, which is strictly weaker than that of Zermelo set theory.
Also, the proof-theoretic strength of $\CZF$ with Powerset in terms of classical theories is provided by Rathjen \cite{Rathjen2012Power}: He showed that the proof-theoretic strength of $\CZF$ with Powerset is equal to a Power-Kripke set theory $\KP(\mathcal{P})$, $\KP$ with Separation and Collection for $\Delta_0^\mathcal{P}$-formulas, which is strictly stronger than Zermelo set theory. For the precise definition of terminologies, see \cite{Rathjen2012Power}.

\subsection{Second-order set theory}
We will define two types of second-order set theories as we classically did. There are two second-order extensions of $\ZF$, namely \emph{G\"odel-Bernays set theory} $\GB$ and \emph{Kelly-Morse set theory} $\KM$. Some literatures may use different names (like $\mathsf{NBG}$ for $\GB$, $\mathsf{MK}$ for $\KM$) but we follow terminologies in \cite{WilliamsPhD}. The main difference between $\GB$ and $\KM$ is while $\GB$ cannot use second-order formulas to formulate new classes, $\KM$ can. 

The following definition defines the constructive analogue of $\GB$:
\begin{definition}
    Second-order set theories are defined over a two-sorted language, each will mean sets and classes. We use $\forall^0$ and $\exists^0$ to quantify sets, and $\forall^1$ and $\exists^1$ to quantify classes.

    Axioms of $\CGB$ are first-order axioms of $\CZF$ with the following second-order axioms:
    \begin{itemize}
        \item Class Extensionality: two classes are equal if they have the same set members. Formally,
			\begin{equation*}
				\forall^1 X,Y [X=Y\lr \forall^0x (x\in X \lr x\in Y)].
			\end{equation*}
			
		\item Elementary Comprehension: if $\phi(x,p,C)$ is a first-order formula with a class parameter $C$, then there is a class $A$ such that $A=\{x\mid \phi(x,p,C)\}$. Formally,
			\begin{equation*}
				\forall^0p\forall^1C\exists^1 A \forall^0x [x\in A\lr \phi(x,p,C)].
			\end{equation*}
			
		\item Class Set Induction: if $A$ is a class, and if we know a set $x$ is a member of $A$ if every element of $x$ belongs to $A$, then $A$ is the class of all sets. Formally,
			\begin{equation*}
				\forall^1A \big[ [\forall^0 x(\forall^0y\in x (y\in A)\to x\in A)]\to \forall^0x (x\in A) \big].
			\end{equation*}
			
		\item Class Strong Collection: if $R$ is a class multi-valued function from a set $a$ to the class of all sets, then there is a set $b$ which is an `image' of $a$ under $R$. Formally,
			\begin{equation*}
				\forall^1 R\forall^0a[R \colon a\rrarrows V \to \exists^0b(R \colon a\lrlrarrows b)].
			\end{equation*}
	\end{itemize}

    Axioms of $\IGB$ are obtained from $\CGB$ by adding axioms of $\IZF$ with the following axiom:
    \begin{itemize}
        \item Class Separation: for any class $A$ and a set $a$, $A\cap a$ is a set. Formally,
        \begin{equation*}
            \forall^1 A \forall^0 a \exists^0 b \forall^0 x [x\in b \lr (x\in A\land x\in a)]
        \end{equation*}
    \end{itemize}
    and putting Class Collection instead of Class Strong Collection:
    \begin{itemize}
        \item Class Collection: if $R$ is a class multi-valued function from a set $a$ to the class of all sets, then there is a set $b$ such that $a$ `projects' to $b$ by $R$. Formally,
			\begin{equation*}
				\forall^1 R\forall^0a[R \colon a\rrarrows V \to \exists^0b(R \colon a\rrarrows b)].
			\end{equation*}
	\end{itemize}
\end{definition}
The reader should be cautioned that Elementary Comprehension does not tell anything about forming sets. That is, $\{x\in a\mid \phi(x)\}$ is a class in general unless $\phi$ is $\Delta_0$. Also, as $\IZF$ implies Strong Collection, $\IGB$ also proves Class Strong Collection from Class Collection:
\begin{lemma}
    $\IGB$ proves Class Strong Collection.
\end{lemma}
\begin{proof}
    It follows the usual proof for Strong Collection over $\IZF$: Suppose that $R\colon a\rrarrows V$ is a multi-valued function and let $b$ be a set such that $R\colon a\rrarrows b$. Then take $c=\{y\in b\mid \exists x\in a (\lag x,y\rag\in R)\}$, which exists by Class Separation. Then $R\colon a\lrlrarrows c$.
\end{proof}

The constructive analogues $\CKM$ and $\IKM$ of $\KM$ are defined from $\CGB$ and $\IGB$ by adding a stronger class formation rule:
\begin{definition}
    $\CKM$ and $\IKM$ are obtained from $\CGB$ and $\IGB$ by adding the following axiom:
    \begin{itemize}
        \item Full Comprehension: if $\phi(x,p,C)$ is a \emph{second-order} formula with a class parameter $C$, then there is a class $A$ such that $A=\{x\mid \phi(x,p,C)\}$.
    \end{itemize}
\end{definition}

Also, the following alternative form of Class Collection eliminates the need to define ordered pairs for its formulation. It is useful when we work with formulas directly, so we want to avoid additional complexity introduced by the definition of ordered pairs.
\begin{proposition}[$\IGB$ without Class Collection]
\label{Proposition: ClassStrongCollection-noOP}
    Class Collection is equivalent to the following claim: for any formula $\phi(x,y,p, A)$ with no class quantifiers,
    \begin{equation*}
		\forall^1 A\forall^0p\forall^0a[ (\forall^0x\in a\exists^0y \phi(x,y,p,A))\to \exists^0b(\forall^0x\in a\exists^0y\in b \phi(x,y,p,A))].
	\end{equation*}
\end{proposition}
\begin{proof}
    For the left-to-right, consider $R=\{\langle x,y\rangle \mid x\in a\land \phi(x,y,p,A)\}$. Then $R\colon a\rrarrows V$, then apply the Class Collection. For the right-to-left, observe that the formula $\langle x,y\rangle \in R$ is a formula without class quantifiers.
\end{proof}
Similarly, we can formulate the following form of Strong Class Collection over $\CGB$:
\begin{proposition}[$\CGB$ without Strong Class Collection] \pushQED{\qed}
    Strong Class Collection is equivalent to the following claim: for any formula $\phi(x,y,p, A)$ with no class quantifiers,
    \begin{equation*}
		\forall^1 A\forall^0p\forall^0a[ (\forall^0x\in a\exists^0y \phi(x,y,p,A))\to \exists^0b(\forall^0x\in a\exists^0y\in b \phi(x,y,p,A) \land \forall^0y\in b\exists^0x\in a \phi(x,y,p,A))]. \qedhere 
	\end{equation*}
\end{proposition}

\section{Other preliminary notions}
\label{Section: Prelim}
In the remaining part of the paper, we will provide a way to reduce $\CKM$ into Second-Order Arithmetic. The main idea is tweaking Lubarsky's interpretation defined in \cite{Lubarsky2006SOA}.
The main idea of Lubarsky's proof is coding sets as well-founded trees combined with Kleene's first realizability algebra. We will not review Lubarsky's proof in full detail, but we still review the necessary details in his paper.

\subsection{Second-Order Arithmetic}
In this paper, we work over Second-Order Arithmetic unless specified. But what is Second-Order Arithmetic? Usually, we characterize Second-Order Arithmetic with first-order Peano arithmetic with Induction and Comprehension for any second-order formulas. In this paper, we additionally assume choice axioms over Second-Order Arithmetic. Before defining the choice schemes, let us mention that we denote $(Y)_i = \{m\mid (i,m)\in Y\}$ to mention a coded family of sets.

\begin{definition}
    $\mathsf{\Sigma^1_\infty\mhyphen AC}$ is the following assertion for any second-order formula $\phi(n,X)$:
    \begin{equation*}
        \forall n\exists X\ \phi(n,X)\to\exists Z \forall n \phi(n,(Z)_n).
    \end{equation*}
    $\mathsf{\Sigma^1_\infty\mhyphen DC}$ is the following assertion for each second-order formula $\phi(n,X,Y)$:
    \begin{equation*}
        \forall n\forall X \exists Y\ \phi(n,X,Y)\to \exists Z\forall n\ \phi(n,(Z)^n,(Z)_n),
    \end{equation*}
    where $(Z)^n = \{(i,m)\in Z \mid i<n\}$
\end{definition}
It is known that adding $\mathsf{\Sigma^1_\infty\mhyphen AC}$ and $\mathsf{\Sigma^1_\infty\mhyphen DC}$ to full Second-Order Arithmetic does not change its proof-theoretic strength. For its proof, see Theorem VII.6.16 of \cite{Simpson2009}. On the other hand, \cite{FriedmanGitmanKanovei2019} showed that  $\mathsf{\Sigma^1_\infty\mhyphen AC}$ does not suffice to prove a fragment of $\mathsf{\Sigma^1_\infty\mhyphen DC}$.

\subsection{Finite sequences and trees}
It is known that we can code a finite sequence of natural numbers into a single natural number, and we denote a code of the sequence $a_0,a_1,\cdots,a_n$ as $\vec{a}=\lag a_0,a_1,\cdots,a_n\rag$. For a finite sequence $\vec{a}=\lag a_0,a_1,\cdots,a_n\rag$, we denote its length as $\len(\vec{a})$, and its restriction $\lag a_0,\cdots,a_{k-1}\rag$ as $\vec{a}\restricts k$.
For two finite sequences $\vec{a}$ and $\vec{b}$, we denote their concatenation as $\vec{a}^\frown \vec{b}$, whose length is $\len\vec{a}+\len\vec{b}$.

Now we can code trees in terms of sets of finite sequences with certain conditions:
\begin{definition}
    A set $S$ is a tree if it satisfies the following conditions:
    \begin{enumerate}
        \item $\lag\rag\in S$,
        \item If $\vec{a}\in S$ and $k\le\len(\vec{a})$, then $\vec{a}\restricts k\in S$.
    \end{enumerate}
\end{definition}

For a natural number $n$ and a tree $S$, we define $S\downarrow n$ by the set $\{\vec{a}\mid \lag n\rag^\frown \vec{a}\in S\}$. In general, if $\vec{a}$ is a finite sequence of natural numbers, we define $S\downarrow \vec{a}$ by $\{\vec{b}\mid \vec{a}^\frown\vec{b}\in S\}$. Also, $\immd(S)$ is the set of all immediate successors of the top node $\lag\rag$ of $S$, that is,
\begin{equation*}
    \immd(S) = \{a \mid \lag a\rag\in S\}.
\end{equation*}
A tree $S$ is \emph{well-founded} if it has no infinite path over $S$, that is,
\begin{equation*}
    \lnot\exists f \forall n[\lag f(0),\cdots, f(n)\rag\in S].
\end{equation*}
We denote that tree $S$ is well-founded as $\WF(S)$. It is known that full Second-Order Arithmetic with the help of $\mathsf{\Sigma^1_\infty\mhyphen DC}$ satisfies $\mathsf{\Pi^1_\infty\mhyphen TI}$, which is equivalent over full Second-Order Arithmetic to the assertion that all well-founded relations satisfy transfinite induction for arbitrary second-order formulas.
\begin{proposition}[$\mathsf{\Pi^1_\infty\mhyphen TI}$] \pushQED{\qed} \label{Proposition: CKM-Sect1-TIinfty}
    Second-Order Arithmetic proves the following: If $\prec$ is a well-founded relation over the domain $D$, and if $\phi(x)$ is a second-order formula, then the following holds:
    \begin{equation*}
        \forall x\in D[(\forall y\prec x \phi(y)) \to\phi(x)]\to \forall x\in D\ \phi(x) \qedhere 
    \end{equation*}
\end{proposition}
Now let $S$ be a well-founded tree and for $\vec{a},\vec{b}\in S$, define $\vec{a}\prec_S \vec{b}$ if $\vec{a}$ is an immediate successor of $\vec{b}$, that is, $\exists n (\vec{a}=\vec{b}^\frown\lag n\rag)$. Then $\prec_S$ forms a well-founded relation over $S$, so we have the following:
\begin{corollary}\pushQED{\qed} \label{Corollary: CKM-Sect1-TreeTI}
    Suppose that $S$ is a well-founded tree and $\phi(x)$ is a second-order statement. If we have
    \begin{equation*}
        \forall \vec{a}\in S [\forall \vec{a}^\frown \lag n\rag \in S\ \phi( \vec{a}^\frown \lag n\rag)]\to \phi(\vec{a}),
    \end{equation*}
    then we have $\forall\vec{a}\in S\phi(\vec{a})$.
    \qedhere 
\end{corollary}

Also, we can prove the following simultaneous induction over two well-founded trees:
\begin{corollary}\label{Corollary: CKM-Sect1-TreeTI2}
    Suppose that $S$ and $T$ are well-founded trees and $\phi(x,y)$ is a second-order statement. If we have
    \begin{equation*}
        \forall \vec{a}\in S\forall \vec{b}\in T [\forall \vec{a}^\frown \lag n\rag \in S\forall \vec{b}^\frown \lag m\rag\in T \ \phi( \vec{a}^\frown \lag n\rag,\vec{b}^\frown \lag m\rag)]\to \phi(\vec{a},\vec{b}),
    \end{equation*}
    then we have $\forall\vec{a}\in S\forall \vec{b}\in T\phi(\vec{a},\vec{b})$.
\end{corollary}
\begin{proof}
    Consider the relation pairwise comparison relation $\prec_S\times \prec_T$ given by 
    \begin{equation*}
        (p,q) \mathrel{(\prec_S\times \prec_T)} (r,s)\iff (p\prec_S r)\land (q\prec_T s)
    \end{equation*}
    over $S\times T$. Then we can see that $\prec_S\times \prec_T$ is a well-founded relation over $S\times T$. Thus the desired conclusion follows from \autoref{Proposition: CKM-Sect1-TIinfty}.
\end{proof}

\subsection{Kleene's first combinatory algebra}
Structures known as \emph{partial combinatory algebras} or \emph{applicative structures}  provide a way to handle generalized computable functions in an algebraic manner. 
We define \emph{partial combinatory algebras} in an axiomatic manner:
\begin{definition}
    $\mathsf{PCA}$ is the intuitionistic first-order theory defined over the language with two constant symbols $\mathbf{k}$ and $\mathbf{s}$, with a ternary relational symbol $\mathsf{App}$. We denote $\mathsf{App}(t_0,t_1,t_2)$ as $t_0t_1\simeq t_2$.

    In practice, we represent statements over $\mathsf{PCA}$ in terms of \emph{partial terms} with the symbol $\simeq$. Partial terms are (non-associative) juxtapositions of terms over $\mathsf{PCA}$, and here $t_0t_1\cdots t_n$ means $(\cdots(t_0t_1)t_2\cdots)$.
    
    We also use the following notations: If $s$ and $t$ are partial terms and $t$ is not a variable, then define
    \begin{equation*}
        s\simeq t\iff \forall y (s\simeq y\leftrightarrow t\simeq y).
    \end{equation*}
    If $a$ is a free variable, define
    \begin{equation*}
        s\simeq a \iff
        \begin{cases}
            s=a & \text{If $s$ is a term of $\mathsf{PCA}$,}\\
            \exists x\exists y [s_0\simeq x\land s_1\simeq y\land xy\simeq a] & \text{if $s$ is of the form $(s_0s_1)$.}
        \end{cases}
    \end{equation*}
    Finally, $t\downarrow$ means $\exists y (t\simeq y)$.
    Now define the axioms of $\mathsf{PCA}$ as follows:
    \begin{enumerate}
        \item $ab\simeq c_0$ and $ab\simeq c_1$ implies $c_0=c_1$,
        \item $(\mathbf{k}ab)\downarrow$ and $\mathbf{k}ab\simeq a$,
        \item $(\mathbf{s}ab)\downarrow$ and $\mathbf{s}abc \simeq ac(bc)$.
    \end{enumerate}
\end{definition}

The structure of $\mathsf{PCA}$ looks simple, but it is known that $\mathsf{PCA}$ is strong enough to express lambda terms, and enough to prove the recursion theorem:
\begin{lemma}[{\cite[\S2.2]{Beeson1985}}]
    \pushQED{\qed}
    Let $t$ be a partial term and $x$ be a variable. Then we can construct a term $\lambda x.t$ whose free variables are that of $t$ except for $x$, such that $\mathsf{PCA}$ proves $\lambda x.t\downarrow$ and $(\lambda x.t)u\simeq t[x/u]$ for any partial terms $u$. \qedhere 
\end{lemma}

\begin{lemma}[Recursion Theorem, {\cite[\S2.8]{Beeson1985}}]
    \pushQED{\qed}
    We can find a partial term $\mathbf{r}$ such that $\mathsf{PCA}$ proves $\mathbf{r}x\downarrow$ and $\mathbf{r}xy\simeq x(\mathbf{r}x)y$. \qedhere
\end{lemma}

It is known that Turing machines form a partial combinatory algebra, also known as \emph{Kleene's first algebra}. In general, for a given set $X$, the collection of all $X$-computable functions $\{e\}^X$ form a partial combinatory algebra.
\begin{proposition}\pushQED{\qed}
    For a set $X\subseteq\mathbb{N}$ and $a,b,c\in \mathbb{N}$, let us define $ab\simeq c$ if $\{a\}^X(b)=c$. Then we get a partial combinatory algebra if we appropriately interpret $\mathbf{s}$ and $\mathbf{k}$.
    Furthermore, we can interpret $\mathbf{s}$ and $\mathbf{k}$  uniform to $X$.
    \qedhere 
\end{proposition}

It is known that we can conservatively extend $\mathsf{PCA}$ so that it includes terms for pairing functions, projections, and natural numbers. However, we only work with Kleene's first algebra, so such an extension in full generality is unnecessary for our purpose. Despite that, we should introduce some notation under the virtue of this extension for later use:
\begin{definition}
    We write $e^Xa$ instead of $\{e\}^X(a)$ if the notation does not cause any confusion. Especially, if $X=\emptyset$, we write $ea$ instead of $e^\emptyset a$.
    Also, we use $\mathbf{p}xy$, $(x)_0$, $(x)_1$ to denote a pairing function and the corresponding projection functions.
\end{definition}

For more details on $\mathsf{PCA}$, see \cite{DalenTroelstraII} or \cite{Beeson1985}.

\subsection{Type structure over Kleene's first realizability algebra}
In this subsection, we will define an internal type structure over Kleene's first realizability algebra by following \cite{Rathjen2005Brouwer}.
\begin{definition}
    Let us define the types and their elements recursively. We call a set of elements of a type $A$ its \emph{extension}, which will be denoted by $\ext(A)$.
    \begin{enumerate}
        \item $\sfN$ is a type and $\ext(\sfN) = \bbN$.
        \item For each $n\in\bbN$, $\sfN_n$ is a type and $\ext(\sfN_n) = \{k\mid k<n\}$.
        \item If $A$ and $B$ are types, then $A+B$ is a type with
        \begin{equation*}
            \ext(A+B) = \{\bfp 0x\mid x\in \ext(A)\}\cup \{\bfp 1y\mid y\in \ext(B)\}.
        \end{equation*}
        \item If $A$ is a type and for each $x\in \ext(A)$, $Fx$ is a type, where $F$ is a member of Kleene's first realizability algebra.
        Then $\prod_{x:A}Fx$ is a type with 
        \begin{equation*}
            \ext({\textstyle \prod_{x:A}Fx}) = \{f\mid \forall x\in \ext(A)\ fx\in \ext(Fx)\}.
        \end{equation*}
        \item If $A$ and $F$ satisfy the same conditions a before, $\sum_{x:A}Fx$ is a type with 
        \begin{equation*}
            \ext({\textstyle \sum_{x:A}Fx}) = \{\bfp x u\mid x\in \ext(A) \land u\in \ext(Fx)\}.
        \end{equation*}
    \end{enumerate}
\end{definition}
By a standard convention, we define $A\to B:= \prod_{x:A}B$ and $A\times B := \sum_{x:A}B$.
We define types and their extensions separately since we want to assign types natural numbers. To do this, let us assign G\"odel number for types:
\begin{definition} \label{Definition: Coding types}
    For each type $A$, let us define its G\"odel number $\ulcorner A\urcorner$ recursively as follows:
    \begin{enumerate}
        \item $\ulcorner \sfN\urcorner = (0)$.
        \item $\ulcorner \sfN_n\urcorner = (1,n)$.
        \item $\ulcorner A+B\urcorner = (2,\ulcorner A\urcorner, \ulcorner B\urcorner)$.
        \item $\ulcorner \prod_{x:A} Fx\urcorner = (3,\ulcorner A\urcorner, F)$.
        \item $\ulcorner \sum_{x:A} Fx\urcorner = (4,\ulcorner A\urcorner, F)$.
    \end{enumerate}
\end{definition}
For notational convenience, we will identify type expressions with their G\"odel numbering. That is, for example, $\prod_{x:A} Fx$ is \emph{equal} to $(3,\ulcorner A\urcorner, F)$.
Lastly, let us define a set constructor in the type theory:
\begin{definition}
    Let $\sup(A,f) = (5,A,f)$.
    $\sfV$ is defined recursively as follows: If $A$ is a (G\"odel code for a) type, and $\forall x\in \ext(A)\ fx\in \sfV$, then $\sup(A,f)\in \sfV$.
\end{definition}

\section{The Proof-theoretic strength of $\CKM$}
\label{Section: Strength of CKM}
The main goal of this section is to prove that $\CKM$ can be interpreted over the full Second-Order Arithmetic.
We will interpret an extension of $\CZF$ that can interpret $\CKM$ instead since interpreting a first-order set theory allows us to use what Lubarsky previously did in \cite{Lubarsky2006SOA}.

\subsection{An intermediate theory}
\begin{definition}
    Let $T$ be an intuitionistic first-order theory comprising the following axioms:
    \begin{enumerate}
        \item Axioms of $\CZF$ with the full Separation.
        \item There is a transitive set $M$ satisfying the second-order $\CZF$ in the following sense: 
        \begin{enumerate}
            \item $M$ is closed under union, bounded separation, and $\omega\in M$.
            \item (Second-order Strong Collection) Let $a\in M$ and $R\colon a\rrarrows M$, then there is $b\in M$ such that $R\colon a\lrlrarrows b$.
            \item (Second-order Subset Collection) Let $a,b\in M$, we can find $c\in M$ such that for every $R\colon a\rrarrows b$ (not necessarily in $M$), we can find $d\in c$ such that $R\colon a\lrlrarrows d$.
        \end{enumerate}
    \end{enumerate}
    $M$ being closed under Pairing is unnecessary since $2\in M$ and $M$ satisfying Second-order Strong Colleciton imply $M$ is closed under Pairing. (cf. \cite[Lemma 11.1.5]{AczelRathjen2010}.)
\end{definition}
That said, the theory $T$ is a combination of $\CZF$ plus the full separation and the existence of a regular set $M$ containing $\omega$ and satisfying the second-order subset collection.
Some materials call such $M$ an \emph{inaccessible set}, but the word inaccessible set may denote a stronger large set additionally satisfying the regular extension axiom. (See discussions around \cite[Definition 3.9]{JeonMatthews2022} for more detail.)

Then we can see that $M$ gives a model of $\CKM$ in the following sense:
\begin{lemma}
    Working over $T$, let $M$ be a transitive model of second-order $\CZF$. Then $(M,\mathcal{P}(M))$ is a model of $\CKM$.
\end{lemma}
\begin{proof}
    The first part of $\CKM$ and the Class Strong Collection hold over $(M,\mathcal{P}(M))$ by the assumption. Since $M$ is transitive, every subset of $M$ satisfies Set Induction, so $(M,\mathcal{P}(M))$ satisfies Class Set Induction. Lastly, $(M,\mathcal{P}(M))$ satisfies Elementary Comprehension since $T$ proves the full Separation.
\end{proof}

Note that the model $(M,\mathcal{P}(M))$ is not a set model since $\mathcal{P}(M)$ may be a proper class. Hence $(M,\mathcal{P}(M))$ being a model of $\CKM$ does not prove the consistency of $\CKM$.

\subsection{Lubarsky's interpretation}
Lubarsky proved in \cite{Lubarsky2006SOA} that the full Second-Order Arithmetic interprets $\CZF$ with the full Separation. We will improve his result by showing his interpretation interprets $T$.

Before providing the improvement, let us review his interpretation. He interpreted sets as well-founded trees, which is akin to \emph{injectively presented realizability model} introduced in Swan's thesis \cite{SwanPhD}. 

\begin{definition}
    We distinguish bounded quantifiers and unbounded quantifiers as syntactically distinct objects. Let us define $e \Vdash \phi$ inductively as follows:
    \begin{itemize}
        \item $e\nVdash \bot$ for every $e$.
        \item $e\Vdash S\in T$ if and only if $(e)_0\in\immd(T)$ and $(e)_1\Vdash S = (e)_0^T$.
        \item $e\Vdash S=T$ if and only if for $a\in \immd(S)$, $(e)_0a\Vdash (S\downarrow a)\in T$ and for $b\in \immd(T)$, $(e)_1b\Vdash (T\downarrow b)\in S$.
        \item $e\Vdash \phi\land\psi$ if and only if $(e)_0\Vdash\phi$ and $(e)_1\Vdash \psi$.
        \item $e\Vdash \phi\lor\psi$ if and only if either $(e)_0=0$ and $(e)_1\Vdash\phi$, or $(e)_0=1$ and $(e)_1\Vdash\psi$.
        \item $e\Vdash\phi\to\psi$ if and only if for all $f$, $f\Vdash\phi$ implies $ef\Vdash\psi$.
        \item $e\Vdash \forall X\in S\phi(X)$ if and only if for all $a\in\immd(S)$, $ea\Vdash \phi(S\downarrow a)$.
        \item $e\Vdash \exists X\in S\phi(X)$ if and only if $(e)_0\in\immd(S)$ and $(e)_1\Vdash \phi(S\downarrow(e)_0)$.
        \item $e\Vdash\forall X\phi(X)$ if and only $\forall X [\Set(X)\to e\Vdash\phi(X)]$.
        \item $e\Vdash\exists X\phi(X)$ if and only $\exists X [\Set(X)\land e\Vdash\phi(X)]$.
    \end{itemize}
    Here $\Set(X)$ is the statement `$X$ is a well-founded tree.'
\end{definition}

Lubrasky proved that the above interpretation gives a model of $\CZF$ with the full Separation.
\begin{theorem}[Lubarsky {\cite{Lubarsky2006SOA}}]
    The above interpretation interprets the axioms of $\CZF$ with the full Separation. 
\end{theorem}

Now let us define $M$ using the type system. Before defining $M$, we introduce a way to turn a member of $\sfV$ into a tree.

\begin{definition}
    Let us define a function $\frakt$ of domain $\sfV$ recursively as follows:
    \begin{equation*}
        \frakt(\sup(A,f)) = \{\lag \rag\} \cup \{\lag x\rag^\frown\sigma\mid x\in\ext(A),\ \sigma\in \frakt(fx)\}.
    \end{equation*}
\end{definition}

Then define
\begin{equation*}
    M = \{\lag\rag\} \cup \{\lag x\rag^\frown \sigma \mid x\in\sfV,\ \sigma\in\frakt(x)\}.
\end{equation*}
We can see that $\immd(\frakt(\sup(A,f)))= \ext(A)$ and $\immd(M) = \sfV$. Also,we have $\frakt(\sup(A,f))\downarrow x =\frakt (f x)$ for $x\in \ext(A)$ by definition. Similarly, we have $M\downarrow x = \frakt(x)$ for $x\in \sfV$.

We will claim that $M$ is the desired set, but we should check first that $M$ is a well-founded tree:
\begin{lemma}
    For each $x\in \sfV$, $\frakt(x)$ is well-founded. Hence $M$ is also well-founded.
\end{lemma}
\begin{proof}
    We prove it by induction on $\sup(A,f)\in\sfV$: Suppose that for every $x\in \ext(A)$, $\frakt(fx)$ is well-founded. Then $\frakt(\sup(A,f))$ is a join of well-founded trees, so it is also well-founded. Similarly, $M$ is also well-founded.   
\end{proof}

Now let us prove we can realize the statement `$M$ is a transitive model of the second-order $\CZF$.' We will divide the proof into parts. The proof is done by constructing realizers, and the type system attached in $M$ helps to find a witness for existential claims from $M$. Constructing the witness in $M$ follows Aczel's type-theoretic interpretation, so the reader should check arguments in \cite{Aczel1982} to see how the forthcoming proof works. 

Throughout the remaining section, $\bfq_1$ and $\bfq_2$ mean recursive functions returning the second and third component of a triple respectively. Especially, we have the following:
\begin{itemize}
    \item $\bfq_1(\sup(A,F))=\bfq_1(\sum_{x:A} Fx)=\bfq_1(\prod_{x:A} Fx)=A$, and
    \item $\bfq_2(\sup(A,F))=\bfq_2(\sum_{x:A} Fx)=\bfq_2(\prod_{x:A} Fx)=F$.
\end{itemize}
Also, let us fix a realizer $\mathsf{i}_r$ realizing $X=X$ for every $X$, whose construction is presented in \cite[Proposition 4.3.3]{SwanPhD}.

\begin{lemma}
    We can realize `$M$ is transitive.'
\end{lemma}
\begin{proof}
    Let us claim that the following holds:
    \begin{equation*}
        \lambda x\lambda y.\bfp ((\bfq_2 x)y) \mathsf{i}_r \Vdash \forall S\in M\forall T\in S (T\in M).
    \end{equation*}
    The above is equivalent to
    \begin{equation*}
        \forall a\in \sfV \forall b\in \immd(\frakt(a)) [\bfp ((\bfq_2 a)b) \mathsf{i}_r\Vdash \frakt(a)\downarrow b \in M].
    \end{equation*}
    Now let $a=\sup(A,F)$, then $\immd(\frakt(a)) = \ext(A)$. Hence the above is equivalent to
    \begin{equation*}
        \forall a\in\sfV \forall b\in \ext(A) [\bfp (Fb)\mathsf{i}_r\Vdash \frakt(Fb)\in M].
    \end{equation*}
    We can see that the above holds since $Fb\in \sfV$ and $\mathsf{i}_r \Vdash \frakt(Fb) = M\downarrow(Fb)$.
\end{proof}

To prove $M$ is closed under Bounded Separation, let us introduce a way to turn the set of realizers for a bounded formula into a type:
\begin{proposition}
    Let $\phi(x_0,\cdots,x_{n-1})$ be a bounded formula with all free variables displayed, and $a_0,\cdots,a_{n-1}\in\sfV$.
    Then we can find a type $\|\phi(a_0,\cdots,a_{n-1})\|$ such that 
    \begin{equation} \label{Formula: Type theoretic interpretation for a bounded formula}
        e\Vdash \phi(\frakt(a_0),\cdots,\frakt(a_{n-1})) \iff e \in \ext(\|\phi(a_0,\cdots,a_{n-1})\|).
    \end{equation}
    Furthermore, the map $a_0,\cdots,a_{n-1}\mapsto \|\phi(a_0,\cdots,a_{n-1})\|$ is recursive.
\end{proposition}
\begin{proof}
    Let us prove it for atomic formulas first. First, let us recursively define
    \begin{equation*}
        \textstyle
        \lVert\sup(A,f)=\sup(B,g)\rVert = \prod_{a:A}\sum_{b:B} \|fa=gb\| \times \prod_{b:B}\sum_{a:A} \|gb=fa\|.
    \end{equation*}
    We can understand the above definition as an effective transfinite recursion (cf. \cite[Theorem 3.2]{Sacks1990HigherRecursion}) as follows: Let us consider the relation $\prec$ over $\sfV$ defined by
    \begin{equation*}
        x \prec y \iff \text{if $y=\sup(A,f)$ then $x=fz$ for some $z\in\ext(A)$.}
    \end{equation*}
    Then $\prec$ is well-founded relation over $\sfV$. Now consider a recursive function $I$ satisfying
    \begin{equation*}
        \{I(e)\}(a)(b) = (4,(3,\bfq_1 a,(4,\bfq_1 b, \lambda xy. \{e\}((\bfq_2a) x)((\bfq_2 b)y) )),(3,\bfq_1 b,(4, \bfq_1 a, \lambda xy. \{e\}((\bfq_2 b)y)((\bfq_2a) x) )))
    \end{equation*}
    (cf. \autoref{Definition: Coding types},) or, alternatively,
    \begin{equation*} \textstyle
        \{I(e)\}(a)(b) = \prod_{x:\bfq_1 a} \sum_{y:\bfq_1 b} \{e\}((\bfq_2a)x)((\bfq_2b)y) \times \prod_{y:\bfq_1 b} \sum_{x:\bfq_1 a} \{e\}((\bfq_2b)y)((\bfq_2a)x)
    \end{equation*}
    By Effective Transfinite Recursion theorem, we can find an index for a recursive function $\mathsf{eq}$ defined over $\sfV$ such that $\{\mathsf{eq}\}$ and $\{I(\mathsf{eq})\}$ are extensionally the same. Then let us take
    \begin{equation*}
        \lVert\sup(A,f) = \sup(B,g)\rVert := \mathsf{eq}(\sup(A,f))(\sup(B,g)).
    \end{equation*}
    Now let us prove that $\|x=y\|$ satisfies \eqref{Formula: Type theoretic interpretation for a bounded formula} by the simultaneous induction on $x,y\in\sfV$. 
    To see this, observe that $e\Vdash \frakt(\sup(A,f)) = \frakt(\sup(B,g))$ are equivalent to the conjunction of the following two statements:
    \begin{enumerate}
        \item $\forall a\in \ext(A) [(e_0a)_0\in \ext(B) \land (e_0a)_1\Vdash \frakt(fa)=\frakt(g(e_0a)_0)]$.
        \item $\forall b\in \ext(B) [(e_1a)_0\in \ext(A) \land (e_1b)_1\Vdash \frakt(gb)=\frakt(f(e_1b)_1)]$.
    \end{enumerate}
    Here we wrote $(e)_0$ and $(e)_1$ by $e_0$ and $e_1$ to reduce parentheses. By the inductive hypothesis, the above statements are equivalent to
    \begin{enumerate}
        \item $\forall a\in \ext(A) [(e_0a)_0\in \ext(B) \land (e_0a)_1 \in \|fa=g(e_0a)_0\|]$.
        \item $\forall b\in \ext(B) [(e_1a)_0\in \ext(A) \land (e_1b)_1\in \|gb=f(e_1b)_1\|]$.
    \end{enumerate}
    We can see that the conjunction of the above two is equivalent to
    \begin{equation*} \textstyle
        \forall a\in \ext(A)[e_0a \in \ext(\sum_{b:B}\|fa=gb\|)] \land \forall b\in\ext(B) [e_1b\in \ext(\sum_{a:A}\|gb=fa\| )].
    \end{equation*}
    It is not too hard to check that the above is equivalent to
    \begin{equation*} \textstyle
        e \in \ext\left(\prod_{a:A}\sum_{b:B}\|fa=gb\| \times \prod_{b:B}\sum_{a:A}\|gb=fa\| \right),
    \end{equation*}
    as we promised.

    For $\in$, let us define
    \begin{equation*} \textstyle
        \|x\in \sup(A,f)\| = \sum_{a:A}\|x=fa\|.
    \end{equation*}
    The map $x,y\mapsto \|x\in y\|$ is recursive since we can express it by
    \begin{equation*}
        \|x\in y\| = (4, \bfq_1 y, \lambda u.\mathsf{eq}(x)((\bfq_2 y)u) ).
    \end{equation*}
    Then we can see that
    \begin{equation*}
        \begin{array}{lcl}
        e \Vdash \frakt(a) \in \frakt(\sup(A,f)) &\iff& (e)_0 \in \ext(A) \land (e)_1 \Vdash \frakt(a) = \frakt(f(e)_0) \\
        &\iff&  \frakt(\sup(A,f)) \\
        &\iff& (e)_0 \in \ext(A) \land (e)_1 \in \|a = f(e)_0\| \\
        &\iff& e \in \ext\left( \sum_{x:A} \|a=fx\|\right).
        \end{array}
    \end{equation*}
    The remaining cases can be proven similarly.
    Constructing $\|\phi\|$ follows the type-theoretic interpretation, and proving the equivalence is a tedious manipulation, so we omit the details.
\end{proof}

\begin{proposition}
    We can realize `$M$ is closed under Bounded Separation.'
\end{proposition}
\begin{proof}
    Let $\phi(x,u)$ be a bounded formula.
    We want to find a realizer $e$ satisfying
    \begin{equation*}
        e\Vdash \forall a,p\in M \exists b\in M [\forall x\in b(x\in a\land \phi(x,p)) \land \forall x\in a(\phi(x,p)\to x\in b)].
    \end{equation*}
    The above is equivalent to the following: For each $a,p\in \sfV$,
    \begin{equation*}
        (eap)_1 \Vdash \forall x\in \frakt((eap)_0)(x\in \frakt(a)\land \phi(x,\frakt(p))) \land \forall x\in \frakt(a)(\phi(x,\frakt(p))\to x\in \frakt((eap)_0)).
    \end{equation*}
    We will mimic the proof of the Bounded Separation under the type-theoretic interpretation to find $b=(eap)_0$: Suppose that $a=\sup(A,F)$. Consider
    \begin{equation*}
        b = \textstyle \sup(\sum_{x:A}\|\phi(Fx,p)\|, \lambda x. F(x)_0).
    \end{equation*}
    Then we can see that the following holds:
    \begin{equation*}
        \lambda x. \bfp(\bfp (F(x)_0)\mathsf{i}_r)(x)_1 \Vdash \forall x\in \frakt(b)(x\in \frakt(a)\land \phi(x,\frakt(p))).
    \end{equation*}
    We also have
    \begin{equation*}
        \lambda x\lambda f.\bfp(\bfp xf)\mathsf{i}_r \Vdash \forall x\in \frakt(a)[\phi(x,\frakt(p))\to x\in \frakt(b)].
    \end{equation*}
    Thus we can pick $(eap)_1 = \bfp(\lambda x. \bfp(\bfp (F(x)_0)\mathsf{i}_r)(x)_1)(\lambda x\lambda f.\bfp(\bfp xf)\mathsf{i}_r)$.
    Hence the following $e$ works:
    \begin{equation*}\textstyle
        e = \lambda a,p. \bfp(\sup(\sum_{x:\bfq_1a}\|\phi((\bfq_2 a)x,p)\|, \lambda x.(\bfq_2 a)(x)_0))
        (\bfp(\lambda x. \bfp(\bfp (F(x)_0)\mathsf{i}_r)(x)_1)(\lambda x\lambda f.\bfp(\bfp xf)\mathsf{i}_r)) \qedhere
    \end{equation*}
\end{proof}

\begin{lemma}
    We can realize `$M$ is closed under Union' and $\omega\in M$.
\end{lemma}
\begin{proof}
    For Union, let us realize the following sentence:
    \begin{equation*}
        \forall a\in M\exists b\in M [\forall x\in a\forall y\in x (y\in b)\land \forall y\in b\exists x\in a (y\in x)].
    \end{equation*}
    For $a=\sup(A,F)\in\sfV$, let us consider 
    \begin{equation*} \textstyle
        b=\sup(\sum_{x:A}\bfq_1(Fx),\lambda z.(\bfq_2(F(z)_0))(z)_1).
    \end{equation*}
    Then we have
    \begin{equation*}
        \lambda xy. \bfp(\bfp xy)\mathsf{i}_r \Vdash \forall x\in a\forall y\in x (y\in b)
    \end{equation*}
    and
    \begin{equation*}
        \lambda y. \bfp(y)_0(\bfp(y)_1\mathsf{i}_r)\Vdash \forall y\in b\exists x\in a (y\in x).
    \end{equation*}
    Hence the following realizes $M$ is closed under Union:
    \begin{equation*} \textstyle
        \lambda a. \bfp(\sup(\sum_{x:\bfq_1 a}\bfq_1((\bfq_2 a)x),\lambda z.(\bfq_2((\bfq_2 a)(z)_0))(z)_1))(\bfp(\lambda xy. \bfp(\bfp xy)\mathsf{i}_r)(\lambda y. \bfp(y)_0(\bfp(y)_1\mathsf{i}_r)) )
    \end{equation*}
    For realizing $\omega\in M$, we prove that there is $\underline{\omega}\in\sfV$ such that $\frakt(\underline{\omega})$ is equal to a tree witnessing the axiom of infinity. By effective recursion, we can find a recursive function $\mathsf{nat}$ satisfying 
    \begin{equation*}
        \mathsf{nat}\ n = \sup(\bbN_n,\mathsf{nat}).
    \end{equation*}
    for all $n$. Then define $\underline{\omega}=\sup(\bbN,\mathsf{nat})$. We can see that
    \begin{equation*}
        \frak{t}(\mathsf{nat}\ n) = \{\lag\rag\} \cup \{\lag k\rag^\frown \sigma \mid k<n\land \sigma\in\frak{t}(\mathsf{nat}\ k)\}
    \end{equation*}
    and $\frakt(\underline{\omega}) = \{\lag n\rag^\frown \sigma\mid n<\omega\land \sigma\in \mathsf{nat}\ n\}$. Hence $\frakt(\underline{\omega})$ is equal to the tree witnessing the axiom of infinity presented in \cite{Lubarsky2006SOA}.
\end{proof}

\begin{proposition}
    We can realize `$M$ satisfies second-order Strong Collection.'
\end{proposition}
\begin{proof}
    Suppose that we are given $a=\sup(A,F)$, a formula $\phi(x,y,R) \equiv (\lag x,y\rag\in R)$ with some parameter well-founded tree $R$, and a realizer $f$ satisfying
    \begin{equation*}
        f \Vdash \forall x\in\frakt(a) \exists y\in M \phi(x,y,R).
    \end{equation*}
    The above is equivalent to
    \begin{equation*}
        \forall x\in \ext(A) [(f(Fx))_0\in\sfV \land (f(Fx))_1 \Vdash \phi(\frakt(Fx), \frakt((f(Fx))_0), R)].
    \end{equation*}
    Take $b=\sup(A, \lambda x.(f(Fx))_0)$. Then we can see that the following two hold:
    \begin{enumerate}
        \item $\lambda x. \bfp x(f(Fx)_1) \Vdash \forall x\in \frakt(a) \exists y\in\frakt(b) \phi(x,y,R)$.
        \item $\lambda x. \bfp x(f(Fx)_1) \Vdash \forall y\in\frakt(b) \exists x\in \frakt(a) \phi(x,y,R)$.
    \end{enumerate}
    Hence we have
    \begin{equation*}
        \bfp(\lambda x. \bfp x(f(Fx)_1))(\lambda x. \bfp x(f(Fx)_1)) \Vdash R\colon \frakt(a)\lrlrarrows \frakt(b).
    \end{equation*}
    Thus we can see that the realizer
    \begin{equation*}\textstyle
        \lambda a, f. \bfp(\sup(\bfq_1 a,\lambda x. f((\bfq_2 a)x)_0))
        (\bfp(\lambda x. \bfp x(f(Fx)_1))(\lambda x. \bfp x(f(Fx)_1)))
    \end{equation*}
    realizes `$M$ satisfies second-order Strong Collection.'
\end{proof}

\begin{proposition}
    We can realize `$M$ satisfies second-order Subset Collection.' 
\end{proposition}
\begin{proof}
    Let us take $\phi(x,y,R)\equiv (\lag x,y\rag\in R)$ as before, and $a,b\in\sfV$. Define
    \begin{equation*}
        c = \sup(\bfq_1 a\to \bfq_1 b, \lambda z. \sup(\bfq_1 a, \lambda x.(\bfq_2 b)(zx)).
    \end{equation*}
    Now assume that $f\Vdash \forall x\in \frakt(a)\exists y\in\frakt(b) \phi(x,y,R)$. If we take $d=\sup(\bfq_1 a, \lambda x.(\bfq_2 b)((fx)_0)$, then 
    \begin{equation*}
        \bfp f(\lambda x.\bfp x(fx)_1) \Vdash \forall x\in \frakt(a) \exists y\in\frakt(d) \phi(x,y,R) \land \forall y\in\frakt(d) \exists x\in \frakt(a) \phi(x,y,R).
    \end{equation*}
    Furthermore, we have
    \begin{equation*}
        \bfp(\lambda x.(\bfq_2b)(fx)_0)\mathsf{i}_r \Vdash \frakt(d)\in \frakt(c).
    \end{equation*}
    Thus we can construct a realizer that realizes $M$ satisfies second-order Collection.
\end{proof}

By combining all the realizers we have constructed, we have
\begin{corollary} \pushQED{\qed}
    We can realize $M$ is a transitive model of second-order $\CZF$. Hence Lubarsky's interpretation satisfies the intermediate theory $T$. \qedhere 
\end{corollary}

\begin{theorem} 
    The proof-theoretic strength of $\CKM$ is that of Second-Order Arithmetic.
\end{theorem}
\begin{proof}
    For the one direction, we proved that the intermediate theory $T$ interprets $\CKM$, and the full Second-Order Arithmetic interprets the intermediate theory.
    For the other direction, it is easy to prove that $\CKM$ interprets intuitionistic full Second-Order Arithmetic, by interpreting second-order predicates over $\mathbb{N}$ to \emph{subclasses} of $\omega$ over $\CKM$. It finishes the proof since the intuitionistic full Second-Order Arithmetic interprets classical full Second-Order Arithmetic via a double-negation translation.
\end{proof}

As a final note, it is possible to add other type constructors, like $\mathsf{W}$-types, into the type system defining $\mathsf{V}$. Adding additional constructors would make $M$ in the intermediate theory a model of $\CZF$ with large set axioms, and especially, we can interpret $\CKM$ with the regular extension axiom from the Second-Order Arithmetic.

\printbibliography

\end{document}